\documentclass[a4paper,11pt, reqno]{amsart}
\usepackage{latexsym}
\usepackage{amssymb,amsfonts,amsmath,mathrsfs}
\begin{document}
\title{Large gaps between consecutive zeros of the Riemann zeta-function}
\author{H. M. Bui}
\address{Mathematical Institute, University of Oxford, OXFORD, OX1 3LB}
\email{hung.bui@maths.ox.ac.uk}

\thanks{The author is supported by the EPSRC Postdoctoral Fellowship grant EP/F041748/1.}

\begin{abstract}
Combining the mollifiers, we exhibit other choices of coefficients that improve the results on large gaps between the zeros of the Riemann zeta-function. Precisely, assuming the Generalized Riemann Hypothesis (GRH), we show that there exist infinitely many consecutive gaps greater than $3.033$ times the average spacing. 
\end{abstract}
\maketitle

\section{Introduction}

Assuming the Riemann Hypothesis (RH), we can write the nontrivial zeros of the Riemann zeta-function as $\rho=\tfrac{1}{2}+i\gamma$, where $\gamma\in\mathbb{R}$. For $0<\gamma\leq\gamma'$ two consecutive ordinates of zeros, we define the normalized gap
\begin{equation*}
\delta(\gamma)=(\gamma'-\gamma)\frac{\log\gamma}{2\pi}.
\end{equation*}
It is a well-known theorem that the number of nontrivial zeros of $\zeta(s)$ with ordinates in $[0,T]$ is $\frac{1}{2\pi}T\log T+O(T)$. Hence on average $\delta(\gamma)$ is $1$. In 1973, by studying the pair correlation of the zeros of the Riemann zeta-function, Montgomery \textbf{\cite{M1}} suggested that there exist arbitrarily large and small gaps between consecutive zeros of $\zeta(s)$. That is to say
\begin{equation*}
\lambda=\limsup_{\gamma}\delta(\gamma)=\infty\quad\textrm{and}\ \ \mu=\liminf_{\gamma}\delta(\gamma)=0,
\end{equation*}
where $\gamma$ runs over all the ordinates of the zeros of the Riemann zeta-function.

In this article, we will focus only on the large gaps. Our main theorem is

\newtheorem{theo}{Theorem}[section]\begin{theo}
Assuming GRH. Then we have $\lambda>3.033$.
\end{theo}

Selberg \textbf{\cite{S}} remarked that he could prove $\lambda>1$. Assuming RH, Mueller \textbf{\cite{M}} showed that $\lambda>1.9$, and later, by a different approach, Montgomery and Odlyzko \textbf{\cite{MO}} obtained $\lambda>1.9799$. The work of Mueller \textbf{\cite{M}} is based on the following idea.

Let $H:\mathbb{C}\rightarrow\mathbb{C}$ and consider the following functions
\begin{equation*}
\mathscr{M}_1(H,T)=\int_{T}^{2T}|H({\scriptstyle{\frac{1}{2}}}+it)|^2dt
\end{equation*}
and
\begin{equation*}
\mathscr{M}_2(H,T;c)=\int_{-c/L}^{c/L}\sum_{T\leq\gamma\leq 2T}|H({\scriptstyle{\frac{1}{2}}}+i(\gamma+\alpha))|^2d\alpha,
\end{equation*}
where $L=\log\frac{T}{2\pi}$. One notes that if
\begin{equation}\label{51}
h(c):=\frac{\mathscr{M}_2(H,T;c)}{\mathscr{M}_1(H,T)}<1,
\end{equation}
then $\lambda>c/\pi$, and if $h(c)>1$, then $\mu<c/\pi$.

Mueller \textbf{\cite{M}} applied this idea to $H(s)=\zeta(s)$. Using $H(s)=\sum_{n\leq T^{1-\varepsilon}}d_{2.2}(n)n^{-s}$, Conrey, Ghosh and Gonek \textbf{\cite{CGG1}} deduced that $\lambda>2.337$. Here $d_r(n)$ is the coefficient of $n^{-s}$ in the Dirichlet series of $\zeta(s)^r$. Later, assuming GRH, they applied to $H(s)=\zeta(s)\sum_{n\leq T^{1/2-\varepsilon}}n^{-s}$ and obtained $\lambda>2.68$ \textbf{\cite{CGG2}}. By considering a more general mollifier
\begin{equation*}
H(s)=\zeta(s)\sum_{n\leq y}\frac{d_{r}(n)P[n]}{n^s},
\end{equation*}
where $y=T^{1/2-\varepsilon}$ and $P[n]=P(\frac{\log y/n}{\log y})$, Ng \textbf{\cite{Ng}} improved that result to $\lambda>3$. In the last two papers, the assumption of GRH is necessary in order to estimate the discrete mean value over the zeros in $\mathscr{M}_2(H,T;c)$. In connection to this work, we also mention a result of Hall \textbf{\cite{H}}, who showed that $\lambda>2.6306$. The results in Hall's paper are actually unconditional, but a lower bound for $\lambda$ can only be obtained if the Riemann Hypothesis is assumed.  

As an extension of Mueller's idea, we are going to use
\begin{equation*}
H(s)=H_{1}(s)+\zeta(s)H_{2}(s),
\end{equation*}
where
\begin{equation*}
H_1(s)=\sum_{n\leq y}\frac{d_{r+1}(n)P_1[n]}{n^s}\qquad\textrm{and}\qquad H_2(s)=\sum_{n\leq y}\frac{d_{r}(n)P_2[n]}{n^s}.
\end{equation*}
Here $y=T^{\vartheta}$, $0<\vartheta\leq1\leq r$, and $P_1[n]=P_1(\frac{\log y/n}{\log y})$, $P_2[n]=P_2(\frac{\log y/n}{\log y})$, where $P_1(x)$, $P_2(x)$ are two polynomials which will be specified later.

\newtheorem{rem}{Remark}[section]\begin{rem}
\emph{It is not clear how to choose some ``good'' $r$, $P_1(x)$ and $P_2(x)$ to obtain the best result the method would give. It is probable that with a better choice of coefficients our theorem can be significantly improved. Nevertheless, our primary goal here is to exhibit a more general mollifier that could improve the work of \textbf{\cite{M}},\textbf{\cite{CGG2}},\textbf{\cite{Ng}}.}
\end{rem}

The work was commented while the author was visiting the University of Rochester. The author would like to thank Professor Steve Gonek for his support and encouragement during that time. Thanks also go to Professor Micah Milinovich for various stimulating discussions.

\section{Main lemmas}

We state our various lemmas concerning the ``square'' terms and ``cross' terms, which come up in the evaluations of $\mathscr{M}_1(H,T)$ and $\mathscr{M}_2(H,T;c)$.

\newtheorem{lemm}{Lemma}[section]\begin{lemm}
Suppose $0<\vartheta<\tfrac{1}{2}$. We have
\begin{equation*}
\int_{T}^{2T}|H_1(\tfrac{1}{2}+it)|^2dt\sim \frac{a_{r+1}T(\log y)^{(r+1)^2}}{\Gamma((r+1)^{2})}\int_{0}^{1}(1-x)^{(r+1)^2-1}P_{1}(x)^2dx,
\end{equation*}
where 
\begin{equation*}
a_{r}=\prod_{p}\bigg(\bigg(1-\frac{1}{p}\bigg)^{r^{2}}\sum_{n\geq0}\frac{d_{r}(p^n)^2}{p^{n}}\bigg).
\end{equation*}
\end{lemm}

The ``cross'' term of $\mathscr{M}_1(H,T)$ is given by

\begin{lemm}
Suppose $0<\vartheta<\tfrac{1}{2}$. We have
\begin{eqnarray*}
\int_{T}^{2T}\zeta(\tfrac{1}{2}+it)H_{1}(\tfrac{1}{2}-it)H_{2}(\tfrac{1}{2}+it)dt\sim\frac{a_{r+1}T(\log y)^{(r+1)^2}}{\Gamma(r+1)\Gamma(r(r+1))}\int_{0}^{1}(1-x)^{r(r+1)-1}Q_{r}(x)P_{2}(x)dx,
\end{eqnarray*}
where  
\begin{equation*}
Q_{u}(x)=\int_{0}^{x}t^{u}P_{1}(x-t)dt.
\end{equation*}
\end{lemm}

These lemmas are proved in Section 5. The other square term of $\mathscr{M}_1(H,T)$ comes from a theorem of Conrey and Ghosh (cf. Theorem 1 \textbf{\cite{CG}}).

\begin{lemm}
Suppose $0<\vartheta<\tfrac{1}{2}$. We have
\begin{eqnarray*}
\int_{T}^{2T}|\zeta H_2(\tfrac{1}{2}+it)|^2dt\sim\frac{a_{r+1}T(\log y)^{(r+1)^2}}{\Gamma(r)^2\Gamma(r^{2})}\int_{0}^{1}(1-x)^{r^{2}-1}\bigg(\vartheta^{-1}R_{r-1}(x)^2-2R_{r}(x)R_{r-1}(x)\bigg)dx,
\end{eqnarray*}
where
\begin{equation*}
R_{u}(x)=\int_{0}^{x}t^{u}P_{2}(x-t)dt.
\end{equation*}
\end{lemm}

The next two lemmas concern the first ``square'' term and the ``cross'' term in the integrand of $\mathscr{M}_2(H,T;c)$.

\begin{lemm}
Suppose $0<\vartheta<\tfrac{1}{2}$. We have
\begin{eqnarray*}
\sum_{T\leq\gamma\leq 2T}H_1(\rho+i\alpha)H_1(1-\rho-i\alpha)&\sim&\frac{a_{r+1}TL(\log y)^{(r+1)^2}}{2\pi\Gamma((r+1)^{2})}\nonumber\\
&&\!\!\!\!\!\!\!\!\!\!\!\!\!\!\!\!\!\!\!\!\!\!\!\!\!\!\!\!\!\!\!\!\!\!\!\!\!\!\!\!\!\!\!\!\!\!\!\!\!\!\!\!\!\!\!\!\!\!\!\!\!\!\!\!\!\!\!\!\!\!\!\!\!\!\!\!\!\!\!\!\!\int_{0}^{1}(1-x)^{(r+1)^{2}-1}\bigg(P_{1}(x)^2-2\vartheta(r+1)P_1(x)\int_{0}^{x}\cos(\alpha\log y t)P_1(x-t)dt\bigg)dx.
\end{eqnarray*}
\end{lemm}

\begin{lemm}
Suppose $0<\vartheta<\tfrac{1}{2}$. On GRH we have
\begin{eqnarray*}
\sum_{T\leq\gamma\leq 2T}\zeta H_2(\rho+i\alpha)H_1(1-\rho-i\alpha)\sim\frac{a_{r+1}TL(\log y)^{(r+1)^2}}{2\pi\Gamma(r+1)\Gamma(r(r+1))}\int_{0}^{1}(1-x)^{r(r+1)-1}A(r,\vartheta;x)dx,
\end{eqnarray*}
where
\begin{eqnarray*}
A(r,\vartheta;u)&=&(1-(i\alpha L)^{-1})Q_{r}(u)P_2(u)\nonumber\\
&&\qquad-\vartheta(r+1)\int_{0}^{u}y^{i\alpha t}Q_{r}(u-t)P_{2}(u)dt-\vartheta r\int_{0}^{u}y^{-i\alpha t}Q_{r}(u)P_{2}(u-t)dt\\
&&\qquad+\frac{T^{-i\alpha}}{i\alpha L}\int_{0}^{u}t^{r}y^{i\alpha t}P_1(u-t)\bigg(\sum_{n=1}^{r}\binom{r}{n}\frac{(i\alpha\log y)^{n}}{(n-1)!}R_{n-1}(u)+P_2(u)\bigg)dt.
\end{eqnarray*}
\end{lemm}

We prove Lemma 2.4 and Lemma 2.5 in Section 6 and Section 7, respectively. The second ``square'' term is given by Ng (cf. Theorem 2 \textbf{\cite{Ng}}).

\begin{lemm}
Suppose $0<\vartheta<\tfrac{1}{2}$. On GRH we have
\begin{eqnarray*}
\sum_{T\leq\gamma\leq 2T}|\zeta H_2(\rho+i\alpha)|^2\sim\frac{a_{r+1}TL(\log y)^{(r+1)^2}}{\pi\Gamma(r)^2\Gamma(r^{2})}\int_{0}^{1}(1-x)^{r^{2}-1}\Re\bigg({\sum_{j=1}^{\infty}(i\alpha\log y)^jB(r,\vartheta,j;x)}\bigg)dx,
\end{eqnarray*}
where
\begin{eqnarray*}
B(r,\theta,j;u)&=&-\frac{r}{j!}\int_{0}^{u}t^{j}R_{r-1}(u)R_{r-1}(u-t)dt\nonumber\\
&&\!\!\!\!\!\!\!\!\!\!\!\!\!\!\!\!\!\!\!\!\!\!\!\!\!\!\!\!\!\!\!\!\!\!\!+\frac{\theta r}{j!}\int_{0}^{u}t^{j}R_{r}(u)R_{r-1}(u-t)dt+\frac{\theta r}{j!}\int_{0}^{u}t^{j}R_{r-1}(u)R_{r}(u-t)dt\nonumber\\
&&\!\!\!\!\!\!\!\!\!\!\!\!\!\!\!\!\!\!\!\!\!\!\!\!\!\!\!\!\!\!\!\!\!\!\!-\theta\Gamma(r)\sum_{n=-2}^{\min\{j,r-2\}}\frac{(-1)^n\binom{r}{n+2}}{(j-n)!(r+n+1)!}\int_{0}^{u}t^{r-1}(\theta^{-1}-t)^{j-n}R_{r+n+1}(u)P_{2}(u-t)dt.
\end{eqnarray*}
\end{lemm}

\begin{rem}
\emph{It is possible to establish these above lemmas for real $r\geq1$ by using the Selberg-Delange method (cf. Chapter II.5 \textbf{\cite{T}}). However, we are not going to elaborate in this direction here.}
\end{rem}

\begin{rem}
\emph{We note that Lemmas 2.1--2.4 are unconditional. Lemma 2.5 and Lemma 2.6, as mentioned in \textbf{\cite{Ng}}, can probably be proved only assuming the Generalized Lindel\"of Hypothesis by following the work of Conrey, Ghosh and Gonek \textbf{\cite{CGG3}}. Even this assumption may possibly be removed since an upper bound for the sixth moment of Dirichlet $L$-functions $L(s,\chi)$ on average is sufficient for the main theorem in \textbf{\cite{CGG3}}. If so, our Theorem 1 would hold on assuming only the Riemann Hypothesis.}
\end{rem}

In Section 8, we illustrate how our theorem follows from Lemmas 2.1--2.6. Throughout the paper, we denote $L=\log\tfrac{T}{2\pi}$, $e(x)=e^{2\pi ix}$. To facilitate the proofs of some lemmas, we sometimes allow $\alpha\in\mathbb{C}$. However, $\alpha$ is always restricted to $\alpha\ll L^{-1}$. We also assume that $y=T^{\vartheta}$, where $0<\vartheta<1/2$, and $r\geq1$.

\section{Initial manipulations for Lemma 2.5}

By Cauchy's theorem we have
\begin{equation*}
S_{12}=\sum_{T\leq\gamma\leq 2T}\zeta H_2(\rho+i\alpha)H_{1}(1-\rho-i\alpha)=\frac{1}{2\pi i}\int_{\mathscr{C}}\frac{\zeta'}{\zeta}(s-i\alpha)\zeta(s)H_{1}(1-s)H_{2}(s)ds,
\end{equation*}
where $\mathscr{C}$ is the positively oriented rectangle with vertices at $1-a+i(T+\alpha)$, $a+i(T+\alpha)$, $a+i(2T+\alpha)$ and $1-a+i(2T+\alpha)$. Here $a=1+L^{-1}$ and $T$ is chosen so that the distances from $T+\alpha$ and $2T+\alpha$ to the nearest $\gamma$ are $\gg L^{-1}$. Now for $s$ inside or on $\mathscr{C}$ we have
\begin{equation*}
H_1(s),H_2(s)\ll y^{1-\sigma}T^\varepsilon\qquad\textrm{and}\qquad \zeta(s)\ll T^{(1-\sigma)/2+\varepsilon}.
\end{equation*}
Also, for each large $T$, we can choose $T'$ such that $T-2<T'<T$, $T'+\alpha$ is not the ordinate of a zero of $\zeta(s)$ and $\zeta'(\sigma+iT')/\zeta(\sigma+iT')\ll L^2$, uniformly for $-1<\sigma<2$ (cf. \textbf{\cite{D}}). A simple argument using Cauchy's residue theorem then yields that the contribution of the bottom edge of the contour is $\ll yT^{1/2+\varepsilon}$. The same argument holds for the top edge. Hence the contribution from the horizontal lines is $O(yT^{1/2+\varepsilon}$).

We denote the contribution from the right edge by
\begin{equation}\label{3}
J_1(H_1,H_2)=\frac{1}{2\pi i}\int_{a+i(T+\alpha)}^{a+i(2T+\alpha)}\frac{\zeta'}{\zeta}(s-i\alpha)\zeta(s)H_{1}(1-s)H_{2}(s)ds.
\end{equation}
From the functional equation we have
\begin{equation}\label{6}
\frac{\zeta'}{\zeta}(1-s-i\alpha)=\frac{\chi'}{\chi}(1-s-i\alpha)-\frac{\zeta'}{\zeta}(s+i\alpha).
\end{equation}
Hence the contribution from the left edge, by substituting $s$ by $1-s$, is
\begin{eqnarray*}
&&\frac{1}{2\pi i}\int_{a-i(T+\alpha)}^{a-i(2T+\alpha)}\frac{\zeta'}{\zeta}(1-s-i\alpha)\zeta(1-s)H_{1}(s)H_{2}(1-s)ds\nonumber\\
&=&\frac{1}{2\pi i}\int_{a-i(T+\alpha)}^{a-i(2T+\alpha)}\chi(1-s)\bigg(\frac{\chi'}{\chi}(1-s-i\alpha)-\frac{\zeta'}{\zeta}(s+i\alpha)\bigg)\zeta(s)H_{1}(s)H_{2}(1-s)ds\nonumber\\
&=&-\overline{J_3(H_1,H_2)}+\overline{J_2(H_1,H_2)},
\end{eqnarray*}
where
\begin{equation}\label{30}
J_2(H_1,H_2)=\frac{1}{2\pi i}\int_{a+i(T+\alpha)}^{a+i(2T+\alpha)}\chi(1-s)\frac{\zeta'}{\zeta}(s-i\alpha)\zeta(s)H_{1}(s)H_{2}(1-s)ds,
\end{equation}
and
\begin{equation*}
J_3(H_1,H_2)=\frac{1}{2\pi i}\int_{a+i(T+\alpha)}^{a+i(2T+\alpha)}\frac{\chi'}{\chi}(1-s+i\alpha)\zeta(1-s)H_{1}(s)H_{2}(1-s)ds.
\end{equation*}
Thus
\begin{equation*}
S_{12}=J_1(H_1,H_2)-\overline{J_3(H_1,H_2)}+\overline{J_2(H_1,H_2)}+O(yT^{1/2+\varepsilon}).
\end{equation*}
The evaluations of $J_1$, $J_2$ and $J_3$ will be carried out in Section 7.

\section{Auxiliary lemmas}

In this section, we present all the lemmas which we will require for later calculations.  We recall a lemma from \textbf{\cite{CGG3}} (cf. Lemma 2).

\begin{lemm}
Suppose that $A(s)=\sum_{h=1}^{\infty}a(h)h^{-s}$, where $a(h)\ll d_{r_1}(h)(\log h)^{l_1}$ for some non-negative $r_1$ and $l_1$. Also let $B(s)=\sum_{k\leq y}b(k)k^{-s}$, where $b(k)\ll d_{r_2}(k)(\log k)^{l_2}$ for some non-negative $r_2$ and $l_2$. Then we have
\begin{eqnarray*}
&&\frac{1}{2\pi i}\int_{a+iT}^{a+i2T}\chi(1-s)A(s)B(1-s)ds=\sum_{k\leq y}\frac{b(k)}{k}\sum_{kT/2\pi\leq h\leq kT/\pi}a(h)e(-h/k)+O(yT^{1/2+\varepsilon}).
\end{eqnarray*} 
\end{lemm}

\begin{lemm}
For $(h,k)=1$ with $k>0$, we define
\begin{equation*}
L(s,h/k)=\sum_{n=1}^{\infty}\frac{e(\frac{nh}{k})}{n^s}\qquad(\sigma>1).
\end{equation*}
Then $L(s,h/k)$ is regular in the entire complex plane except when $k=1$. For $k=1$ we have $L(s,h/k)=\zeta(s)$ and the function has a simple pole at $s=1$ with residue $1$.
\end{lemm}
The proof of Lemma 4.2 is trivial. The $L$-function defined above is a special case of the Lerch zeta-function. 

\begin{lemm}
For $(h,k)=1$, we define
\begin{equation*}
Q(s,\alpha,h/k)=-\sum_{m,n=1}^{\infty}\frac{\Lambda(n)}{m^sn^{s-i\alpha}}e\bigg(\frac{-mnh}{k}\bigg)\qquad(\sigma>1).
\end{equation*}
Then $Q(s,\alpha,h/k)$ has a meromorphic continuation to the entire complex plane. For $\alpha\ne0$, $Q(s,\alpha,h/k)$ has\\
\emph{(i)} a simple pole at $s=1$ with residue
\begin{equation*}
\left\{ \begin{array}{ll}
\frac{\zeta'}{\zeta}(1-i\alpha) & \textrm{if $k=1$}\\
\frac{-\log p}{p^{(1-i\alpha)\lambda}(1-p^{-1+i\alpha})} & \textrm{if $k=p^\lambda>1$}\\
0 & \textrm{otherwise;}
\end{array} \right.
\end{equation*}
\emph{(ii)} a simple pole at $s=1+i\alpha$ with residue
\begin{equation*}
-\frac{\zeta(1+i\alpha)}{k^{i\alpha}\varphi(k)}\prod_{p|k}(1-p^{i\alpha}).
\end{equation*}
Moreover, on GRH, $Q(s,\alpha,h/k)$ is regular in $\sigma>1/2$ except for these two poles.
\end{lemm}
\begin{proof}
For $\sigma>1$ we have
\begin{eqnarray}\label{18}
Q(s,\alpha,h/k)=\sum_{a=1}^{k}L(s,-ah/k)L(s-i\alpha,a,k)=\sum_{d|k}\sum_{a=1}^{k/d}{\!}^{*}L(s,-ahd/k)L(s-i\alpha,ad,k),
\end{eqnarray}
where $L(s,h/k)$ is the function defined in the previous lemma and
\begin{equation*}
L(s,a,k)=-\sum_{n\equiv a(\textrm{mod}\ k)}\Lambda(n)n^{-s},\qquad(\sigma>1),
\end{equation*}
with $\sum^{*}$ denotes summation over $a$ coprime to $k/d$. It is known that $L(s,a,k)$ has a meromorphic continuation to the entire complex plane and is regular on $\sigma=1$ except for a simple pole at $s=1$ if, and only if, $(a,k)=1$. Also, by Lemma 4.2, $L(s,-ahd/k)$ is regular everywhere except for a simple pole at $s=1$ (when $d=k$). Thus, by \eqref{18}, $Q(s,\alpha,h/k)$ has a meromorphic continuation to the entire complex plane and if $\alpha\ne0$, $Q(s,\alpha,h/k)$ has simple poles at $s=1$ and $s=1+i\alpha$.

From Lemma 4.2, the residue at $s=1$ is
\begin{equation*}
L(1-i\alpha,k,k)=-\sum_{n=1}^{\infty}\frac{\Lambda(kn)}{(kn)^{1-i\alpha}}=\left\{ \begin{array}{ll}
\frac{\zeta'}{\zeta}(1-i\alpha) & \textrm{if $k=1$}\\
\frac{-\log p}{p^{(1-i\alpha)\lambda}(1-p^{-1+i\alpha})} & \textrm{if $k=p^\lambda>1$}\\
0 & \textrm{otherwise.}
\end{array} \right.
\end{equation*}

To evaluate the residue at $s=1+i\alpha$, we note that in \eqref{18}, $L(s-i\alpha,ad,k)$ is regular on $\sigma=1$ unless $d=1$. In the case $d=1$, it has a pole at $s=1+i\alpha$ with residue $-1/\varphi(k)$. Hence the residue of $Q(s,\alpha,h/k)$ at $s=1+i\alpha$ is
\begin{equation*}
-\frac{1}{\varphi(k)}\sum_{a=1}^{k}{\!}^{*}L(1+i\alpha,-ah/k)=-\frac{1}{\varphi(k)}\sum_{n=1}^{\infty}\frac{c_{k}(n)}{n^{1+i\alpha}},
\end{equation*}
where $c_{k}(n)$ is the Ramanujan sum. From Titchmarsh \textbf{\cite{T1}}, this is equal to
\begin{equation*}
-\frac{\zeta(1+i\alpha)}{k^{i\alpha}\varphi(k)}\sum_{d|k}\mu(d)d^{i\alpha}=-\frac{\zeta(1+i\alpha)}{k^{i\alpha}\varphi(k)}\prod_{p|k}(1-p^{i\alpha}).
\end{equation*}
The lemma follows.
\end{proof}

We need a lemma to deal with product of several Dirichlet series (Lemma 3 of \textbf{\cite{CGG2}}).

\begin{lemm}
Suppose that $\mathcal{A}_j(s)=\sum_{n\geq1}\alpha_{j}(n)n^{-s}$ is absolutely convergent for $\sigma>1$, for $1\leq j\leq l$, and that
\begin{equation*}
\mathcal{A}(s)=\sum_{n=1}^{\infty}\frac{\alpha(n)}{n^s}=\prod_{j=1}^{l}\mathcal{A}_j(s).
\end{equation*}
Then for any positive integer $d$, we have
\begin{equation*}
\sum_{n=1}^{\infty}\frac{\alpha(nd)}{n^s}=\sum_{d_1\ldots d_l=d}\prod_{j=1}^{l}\bigg(\sum_{\substack{n\geq1\\(n,\prod_{i<j}d_i)=1}}\frac{\alpha_{j}(nd_j)}{n^s}\bigg).
\end{equation*}
\end{lemm}

The previous three lemmas lead to the following.

\begin{lemm}
Assume GRH. Let $k\in\mathbb{N}$ with $k\leq y$. We define
\begin{equation}\label{8}
Q^{*}(s,\alpha,k)=\sum_{h=1}^{\infty}\frac{a(h)e(-h/k)}{h^s},
\end{equation}
where
\begin{equation}\label{9}
a(h)=-\sum_{\substack{nuv=h\\n\leq y}}d_{r+1}(n)P_{1}[n]\Lambda(u)u^{i\alpha}.
\end{equation}
Then $Q^{*}(s,\alpha,k)$ has an analytic continuation to $\sigma>\tfrac{1}{2}$ except for possible poles at $s=1$ and $s=1+i\alpha$. Moreover we have
\begin{equation*}
Q^{*}(s,\alpha,k)\ll y^{1/2}T^{\varepsilon},
\end{equation*}
for $\tfrac{1}{2}+L^{-1}\leq\sigma\leq a$, $|t|\leq T$ and $|s-1|,|s-1-i\alpha|\gg1$.
\end{lemm}
\begin{proof}
For $\chi$ a character (mod $k$), the Gauss sum $\tau(\chi)$ is given by
\begin{equation*}
\tau(\chi)=\sum_{h=1}^{k}\chi(h)e\bigg(\frac{h}{k}\bigg).
\end{equation*}
It is standard to show that
\begin{equation*}
e\bigg(\frac{-h}{k}\bigg)=\sum_{d|(h,k)}\frac{1}{\varphi(k/d)}\sum_{\chi(\textrm{mod}\ k/d)}\tau(\overline{\chi})\chi\bigg(\frac{-h}{d}\bigg).
\end{equation*}
Inserting this into \eqref{8} leads to
\begin{equation*}
Q^{*}(s,\alpha,k)=\sum_{d|k}\frac{1}{\varphi(k/d)d^s}\sum_{\chi(\textrm{mod}\ k/d)}\tau(\overline{\chi})\overline{\chi}(-d)A(s,d),
\end{equation*}
where
\begin{equation*}
A(s,d)=\sum_{h=1}^{\infty}\frac{a(hd)\chi(hd)}{h^s}\qquad(\sigma>1).
\end{equation*}
By expanding $P_1(x)=\sum_{j\geq0}c_jx^j$ in \eqref{9} we obtain
\begin{equation}\label{12}
Q^{*}(s,\alpha,k)=\sum_{j\geq0}\frac{c_j}{(\log y)^j}Q_{j}^{*}(s,\alpha,k),
\end{equation}
where
\begin{equation}\label{11}
Q_{j}^{*}(s,\alpha,k)=\sum_{d|k}\frac{1}{\varphi(k/d)d^s}\sum_{\chi(\textrm{mod}\ k/d)}\tau(\overline{\chi})\overline{\chi}(-d)\frac{\partial^{j}}{\partial z^j}A(s,d;z)|_{z=0},
\end{equation}
\begin{equation*}
A(s,d;z)=\sum_{h=1}^{\infty}\frac{a_z(hd)\chi(hd)}{h^s},\ \textrm{and}\ a_z(h)=-\sum_{\substack{nuv=h\\n\leq y}}\frac{d_{r+1}(n)y^z\Lambda(u)u^{i\alpha}}{n^z}.
\end{equation*}
Let 
\begin{equation*}
F(s,r,\chi)=\prod_{p|r}(1-\chi(p)p^{-s}).
\end{equation*}
We note that
\begin{equation*}
A(s,1;z)=\bigg(\sum_{n\leq y}\frac{\chi(n)d_{r+1}(n)y^z}{n^{s+z}}\bigg)L(s,\chi)\bigg(-\sum_{u=1}^{\infty}\frac{\chi(u)\Lambda(u)}{u^{s-i\alpha}}\bigg).
\end{equation*}
Hence, by Lemma 4.4,
\begin{equation*}
A(s,d;z)=\sum_{h_1h_2h_3=d}\mathcal{A}_{1}(s,h_1;z)\mathcal{A}_{2}(s,h_2,h_1)\mathcal{A}_{3}(s,h_3,h_1h_2),
\end{equation*}
where
\begin{eqnarray*}
&&\mathcal{A}_{1}(s,h;z)=\sum_{n\leq y/h}\frac{\chi(hn)d_{r+1}(hn)y^z}{n^{s}(hn)^z},\\
&&\mathcal{A}_{2}(s,h,l)=\sum_{(n,l)=1}\frac{\chi(hn)}{n^s}=\chi(h)L(s,\chi)F(s,l,\chi),\\
&&\mathcal{A}_{3}(s,h,l)=-\sum_{(n,l)=1}\frac{\chi(hn)\Lambda(hn)(hn)^{i\alpha}}{n^{s}}.
\end{eqnarray*}
It is obvious that $\mathcal{A}_{1}$ and $\mathcal{A}_{2}$ are regular everywhere except when $\chi$ is principal. In this case $\mathcal{A}_{2}$ has a simple pole at $s=1$. Also, assuming GRH, $\mathcal{A}_{3}$ is regular in $\sigma>1/2$, except for a possible simple pole at $s=1+i\alpha$. Thus, $A(s,d;z)$ is regular in $\sigma>1/2$ with the possible exception of poles at $s=1$ and $s=1+i\alpha$. Hence the required continuation of $Q^{*}(s,\alpha,k)$ follows.

To bound $Q^{*}(s,\alpha,k)$ we will need to bound $A(s,d;z)$. In the considered region we have
\begin{equation*}
\mathcal{A}_{j}(s,h,l)\ll T^{\varepsilon},
\end{equation*}
for $j=2,3$ and if $h,l$ divide $d$ (cf. (3.10) \textbf{\cite{CGG2}}), and (cf. (50) and (54) \textbf{\cite{Ng}})
\begin{equation*}
\mathcal{A}_{1}(s,h;z) \ll \left\{ \begin{array}{ll}
y^{1/2}T^{\varepsilon} & \textrm{if $\chi$ is principal}\\
T^{\varepsilon} & \textrm{otherwise.}
\end{array} \right.
\end{equation*}
Hence in the region under consideration we have
\begin{equation*}
A(s,d;z) \ll \left\{ \begin{array}{ll}
y^{1/2}T^{\varepsilon} & \textrm{if $\chi$ is principal}\\
T^{\varepsilon} & \textrm{otherwise,}
\end{array} \right.
\end{equation*}
uniformly for $|z|\ll L^{-1}$. Applying the Cauchy integral formula with a circle of radius $\asymp L^{-1}$ leads to
\begin{equation*}
\frac{\partial^{j}}{\partial z^j}A(s,d;z)|_{z=0} \ll \left\{ \begin{array}{ll}
y^{1/2}T^{\varepsilon} & \textrm{if $\chi$ is principal}\\
T^{\varepsilon} & \textrm{otherwise.}
\end{array} \right.
\end{equation*}
Combining this with \eqref{11} we obtain
\begin{eqnarray*}
Q_{j}^{*}(s,\alpha,k)&\ll&T^{\varepsilon}\sum_{d|k}\frac{1}{\varphi(k/d)d^{1/2}}\bigg(y^{1/2}|\tau(\chi_0)|+\sum_{\chi\ne\chi_0(\textrm{mod}\ k/d)}|\tau(\chi)|\bigg)\nonumber\\
&\ll&T^{\varepsilon}\bigg((y/k)^{1/2}\sum_{d|k}\frac{d^{1/2}}{\varphi(d)}+k^{1/2}\sum_{d|k}d^{-1}\bigg)\nonumber\\
&\ll&y^{1/2}T^{\varepsilon}.
\end{eqnarray*}
Thus, by \eqref{12} the lemma follows.
\end{proof}

We require the following version of the Landau-Gonek explicit formula \textbf{\cite{G}}.

\begin{lemm}
For $x>1$ we have
\begin{eqnarray*}
\sum_{T\leq\gamma\leq 2T}x^{\rho}&=&-\frac{T}{2\pi}\Lambda(x)+O(x\log(xT)\log\log x)\\
&&\ +O\bigg(\log x\min\bigg(T,\frac{x}{\langle x\rangle}\bigg)\bigg)+O\bigg(L\min\bigg(T,\frac{1}{\log x}\bigg)\bigg),
\end{eqnarray*}
where $\langle x\rangle$ denotes the distance from $x$ to the closest prime power other than $x$ itself, and $\Lambda(x)=\log p$ if $x$ is a positive integral power of a prime $p$ and $\Lambda(x)=0$ otherwise.
\end{lemm}

We also need various lemmas concerning divisor sums and other divisor-like sums. We first introduce some notation which we will use throughout. Let $D_r(n)=D_r(n,1)$, where
\begin{equation*}
D_{r}(n,s):=\bigg(\sum_{m=1}^{\infty}\frac{d_r(mn)}{m^s}\bigg)\zeta(s)^{-r}=\prod_{p^\lambda||n}\bigg(\bigg(1-\frac{1}{p^s}\bigg)^r\sum_{j=0}^{\infty}\frac{d_r(p^{j+\lambda})}{p^{js}}\bigg)\qquad(\sigma>1).
\end{equation*}
We define 
\begin{equation*}
F_\tau(n)=\prod_{p|n}(1+O(p^{-\tau})),
\end{equation*}
for $\tau>0$ and the constant in the $O$-term is implicit and independent of $\tau$. We note that
\begin{equation*}
D_{r}(n,s)\ll d_r(n)F_\tau(n)\qquad(\sigma\geq\tau>0).
\end{equation*}

\begin{lemm}
For $f\in C^{1}([0,1])$, there exists an absolute constant $\tau_0$ such that
\begin{eqnarray}
&&\!\!\!\!\!\!\!\!\!\!\!\!\!\!\!\!\!\!\!\!\!\!\!\!\sum_{n\leq y/k}\frac{d_{r}(kn)f[kn]}{n}=\frac{D_r(k)(\log y)^r}{\Gamma(r)}g_{r-1}[k]+O(d_r(k)F_{\tau_0}(k)L^{r-1}),\nonumber\\
&&\!\!\!\!\!\!\!\!\!\!\!\!\!\!\!\!\!\!\!\!\!\!\!\!\sum_{n\leq y/k}\frac{\Lambda(n)d_{r}(kn)f[kn]}{n^{1-i\alpha}}=rd_r(k)\log y\int_{0}^{\frac{\log y/k}{\log y}}y^{i\alpha t}f\bigg(\frac{\log y/k}{\log y}-t\bigg)dt+O(d_r(k)F_{\tau_0}(k)),\label{5}
\end{eqnarray}
and
\begin{eqnarray*}
\sum_{mn\leq y/k}\frac{\Lambda(n)d_{r}(kmn)f[kmn]}{mn^{1-i\alpha}}&=&\frac{rD_r(k)(\log y)^{r+1}}{\Gamma(r)}\int_{0}^{\frac{\log y/k}{\log y}}y^{i\alpha t}g_{r-1}\bigg(\frac{\log y/k}{\log y}-t\bigg)dt\nonumber\\
&&\qquad\qquad+O(d_r(k)F_{\tau_0}(k)L^{r}),
\end{eqnarray*}
where 
\begin{equation*}
g_{u}(x)=\int_{0}^{x}t^{u}f(x-t)dt.
\end{equation*}
\end{lemm}
\begin{proof}
We note that (cf. Lemma 4 \textbf{\cite{CG}})
\begin{equation*}
\sum_{n\leq y}\frac{d_{r}(kn)}{n}=\frac{D_r(k)(\log y)^r}{\Gamma(r+1)}+O(d_r(k)F_{\tau_0}(k)L^{r-1}),
\end{equation*}
uniformly for all $k$. Hence by Stieltjes integration we have
\begin{eqnarray*}
\sum_{n\leq y/k}\frac{d_{r}(kn)f[kn]}{n}=\frac{D_r(k)}{\Gamma(r)}\int_{1}^{y/k}\frac{(\log \eta)^{r-1}}{\eta}f[k\eta]d\eta+O(d_r(k)F_{\tau_0}(k)L^{r-1}).
\end{eqnarray*}
Substituting $\log \eta/\log y=t$, the first statement of the lemma follows.

We will now only prove the second statement as the last statement is similar. We note that the terms for which $n=p^\lambda$, where $\lambda\geq2$, or $n$ is a prime divisor of $k$ may be included in the error term. So
\begin{equation*}
\sum_{n\leq y/k}\frac{\Lambda(n)d_{r}(kn)f[kn]}{n^{1-i\alpha}}=rd_{r}(k)\sum_{p\leq y/k}\frac{\log p}{p^{1-i\alpha}}f[kp]+O(d_r(k)F_{\tau_0}(k)).
\end{equation*}
By the prime number theorem and Stieltjes integration, the above main term is
\begin{equation*}
rd_{r}(k)\int_{1}^{y/k}\frac{1}{\eta^{1-i\alpha}}f[k\eta]d\eta+O(d_r(k)F_{\tau_0}(k)).
\end{equation*}
We obtain \eqref{5} by the substitution $\log \eta/\log y=t$.
\end{proof}

We need a lemma concerning the size of the function $F_{\tau_0}(n)$ on average. 

\begin{lemm}
For any $\tau_0>0$, we have
\begin{equation*}
\sum_{k\leq y}\frac{d_{r_1}(k)d_{r_2}(k)F_{\tau_0}(k)}{k}\ll L^{r_1r_2}.
\end{equation*}
\end{lemm}
\begin{proof}
We have
\begin{equation*}
F_{\tau_0}(k)\leq\prod_{p|k}(1+Ap^{-\tau_0})=\sum_{n|k}n^{-\tau_0}A^{w(n)}
\end{equation*}
for some $A>0$, where $w(d)$ is the number of prime factors of $d$. Hence
\begin{eqnarray*}
\sum_{k\leq y}\frac{d_{r_1}(k)d_{r_2}(k)F_{\tau_0}(k)}{k}&\ll&\sum_{n\leq y}\frac{A^{w(n)}}{n^{1+\tau_0}}\sum_{k\leq y/n}\frac{d_{r_1}(kn)d_{r_2}(kn)}{k}\nonumber\\
&\ll&L^{r_1r_2}\sum_{n\leq y}\frac{A^{w(n)}d_{r_1}(n)d_{r_2}(n)}{n^{1+\tau_0}}\nonumber\\
&\ll&L^{r_1r_2},
\end{eqnarray*}
since $A^{w(n)}d_{r_1}(n)d_{r_2}(n)\ll n^{\tau_0/2}$ for sufficiently large $n$.
\end{proof}

\begin{lemm}
We have
\begin{equation*}
\sum_{k\leq y}\frac{d_{r}(k)^2}{k}=\frac{a_{r}(\log y)^{r^2}}{\Gamma(r^{2}+1)}+O(L^{r^2-1}),
\end{equation*}
and
\begin{equation*}
\sum_{k\leq y}\frac{D_{r+1}(k)d_{r}(k)}{k}=\frac{a_{r+1}(\log y)^{r(r+1)}}{\Gamma(r(r+1)+1)}+O(L^{r(r+1)-1}).
\end{equation*}
Also let
\begin{equation*}
A(n)=\prod_{p|n}(1-p^{-(1+i\alpha)}).
\end{equation*}
Then 
\begin{equation*}
\sum_{k\leq y}\frac{D_{r+1}(k)d_{r}(k)A(k)}{\varphi(k)}=\frac{a_{r+1}(\log y)^{r(r+1)}}{\Gamma(r(r+1)+1)}+O(L^{r(r+1)-1}).
\end{equation*}
\end{lemm}
\begin{proof}
The first statement is a well-known result. The other two statements can be proved very similarly with minor changes.
\end{proof}

The above lemma leads to

\begin{lemm}
For $f\in C^{1}([0,1])$, we have
\begin{equation*}
\sum_{k\leq y}\frac{d_{r}(k)^2f[k]}{k}=\frac{a_{r}(\log y)^{r^{2}}}{\Gamma(r^{2})}\int_{0}^{1}(1-x)^{r^{2}-1}f(x)dx+O(L^{r^{2}-1}),
\end{equation*}
and
\begin{eqnarray*}
\sum_{k\leq y}\frac{D_{r+1}(k)d_{r}(k)f[k]}{k}=\frac{a_{r+1}(\log y)^{r(r+1)}}{\Gamma(r(r+1))}\int_{0}^{1}(1-x)^{r(r+1)-1}f(x)dx+O(L^{r(r+1)-1}).
\end{eqnarray*}
\end{lemm}
\begin{proof}
These formulae easily follow from Lemma 4.9 and Stieltjes integration. 
\end{proof}

The next lemma is an easy consequence of Lemma 4.7, Lemma 4.8 and Lemma 4.10.

\begin{lemm}
We have
\begin{eqnarray*}
&&\sum_{\substack{h,k\leq y\\h=kn}}\frac{d_{r+1}(h)P_{1}[h]d_{r}(k)P_{2}[k]}{h}\sim\frac{a_{r+1}(\log y)^{(r+1)^2}}{\Gamma(r+1)\Gamma(r(r+1))}\int_{0}^{1}(1-x)^{r(r+1)-1}Q_{r}(x)P_{2}(x)dx,\\
&&\sum_{\substack{h,k\leq y\\h=kn}}\frac{\Lambda(n)d_{r}(h)P[h]d_{r}(k)P[k]}{hn^{-i\alpha}}\sim\frac{ra_r(\log y)^{r^2+1}}{\Gamma(r^2)}\int_{0}^{1}\int_{0}^{x}(1-x)^{r^2-1}y^{i\alpha t}P(x-t)P(x)dtdx,
\end{eqnarray*}
and
\begin{eqnarray*}
\sum_{\substack{h,k\leq y\\h=kmn}}\frac{\Lambda(n)d_{r+1}(h)P_{1}[h]d_{r}(k)P_{2}[k]}{hn^{-i\alpha}}&\sim&\frac{ra_{r+1}(\log y)^{(r+1)^2+1}}{\Gamma(r+1)\Gamma(r(r+1))}\nonumber\\
&&\!\!\!\!\!\!\!\!\!\!\!\!\!\!\!\!\!\!\!\!\!\!\!\!\!\!\!\!\!\!\int_{0}^{1}\int_{0}^{x}(1-x)^{r(r+1)-1}y^{i\alpha t}Q_{r}(x-t)P_2(x)dtdx,
\end{eqnarray*}
where each formula is valid up to a saving of $L$ in the error term.
\end{lemm}

\begin{lemm}
Assume RH. Let
\begin{equation*}
g(k)=\prod_{p|k}(1-p^{i\alpha}).
\end{equation*}
Then we have, for some $\tau_0>0$,
\begin{equation*}
\sum_{k\leq y}\frac{d_{r}(km)g(k)}{\varphi(km)}=\bigg(\sum_{j=0}^{r}\binom{r}{j}\frac{(-i\alpha\log y)^j}{j!}\bigg)\frac{d_{r}(m)}{\varphi(m)}+O\bigg(\frac{d_r(m)F_{\tau_0}(m)}{m}L^{-1}\bigg).
\end{equation*}
\end{lemm}
\begin{proof}
We first consider the generating series of the above sum
\begin{equation*}
H(s,\alpha)=\sum_{k=1}^{\infty}\frac{d_{r}(km)g(k)}{\varphi(km)k^s}.
\end{equation*}
By multiplicativity we have
\begin{eqnarray}\label{47}
H(s,\alpha)&=&\prod_{p}\bigg(\sum_{j=0}^{\infty}\frac{d_r(p^j)g(p^j)}{\varphi(p^j)p^{js}}\bigg)\prod_{p^\lambda||m}\bigg(\frac{\sum_{j=0}^{\infty}d_r(p^{j+\lambda})g(p^{j})/\varphi(p^{j+\lambda})p^{js}}{\sum_{j=0}^{\infty}d_r(p^{j})g(p^{j})/\varphi(p^j)p^{js}}\bigg)\nonumber\\
&=&Z_1(s,\alpha)Z_2(s,\alpha),
\end{eqnarray}
say. We also decompose $Z_1(s,\alpha)$ as
\begin{equation}\label{50}
Z_1(s,\alpha)=\frac{\zeta(1+s)^r}{\zeta(1+s-i\alpha)^r}Z_{11}(s,\alpha),
\end{equation}
where
\begin{eqnarray*}
Z_{11}(s,\alpha)&=&\prod_{p}\bigg[\bigg(1-\frac{1}{p^{1+s}}\bigg)^r\bigg(1-\frac{1}{p^{1+s-i\alpha}}\bigg)^{-r}\bigg(\sum_{j=0}^{\infty}\frac{d_r(p^j)g(p^j)}{\varphi(p^j)p^{js}}\bigg)\bigg]\\
&=&\prod_{p}\bigg[\bigg(1-\frac{1}{p^{1+s}}\bigg)^r\bigg(1-\frac{1}{p^{1+s-i\alpha}}\bigg)^{-r}\bigg(1+\frac{r(1-p^{i\alpha})}{(p-1)p^s}+\sum_{j=2}^{\infty}\frac{d_r(p^j)g(p^j)}{(1-p^{-1})p^{j(s+1)}}\bigg)\bigg].
\end{eqnarray*}
The product for $Z_{11}(s,\alpha)$ is absolutely and uniformly convergent for $\sigma\geq-1/3$, $|\alpha|\ll L^{-1}$. Hence it represents a bounded analytic function of $s$ and $\alpha$ in that region. We next consider $Z_2(s,\alpha)$. We have, for $s=\sigma+it$,
\begin{eqnarray}\label{46}
\sum_{j=0}^{\infty}\frac{d_r(p^j)g(p^j)}{\varphi(p^j)p^{js}}=\bigg(1-\frac{1}{p^{s+1}}\bigg)^{-r}\bigg(1-\frac{r}{p^{s+1-i\alpha}}+O(p^{-2-\sigma})\bigg).
\end{eqnarray}
Furthermore, we note that (cf. \textbf{\cite{Ng}})
\begin{eqnarray*}
\sum_{j=0}^{\infty}\frac{d_r(p^{j+\lambda})}{p^{j(s+1)}}=\bigg(1-\frac{1}{p^{s+1}}\bigg)^{-r-1}d_r(p^\lambda)(1+O(p^{-1-\sigma})).
\end{eqnarray*}
It is then standard to verify that
\begin{eqnarray*}
\sum_{j=0}^{\infty}\frac{d_r(p^{j+\lambda})g(p^{j})}{\varphi(p^{j+\lambda})p^{js}}=\bigg(1-\frac{1}{p^{s+1}}\bigg)^{-r}\frac{d_r(p^\lambda)}{p^\lambda}(1+O(p^{-1-\sigma})).
\end{eqnarray*}
Combining this with \eqref{47} and \eqref{46} we obtain
\begin{eqnarray}\label{49}
|Z_2(s,\alpha)|\leq\prod_{p^\lambda||m}\frac{d_r(p^\lambda)}{p^\lambda}(1+O(p^{-1-\sigma}))|1-rp^{-s-1+i\alpha}|^{-1}\leq\frac{d_r(m)F_{\tau_0}(m)}{m},
\end{eqnarray}
for some positive constant $\tau_0$, in the region $\sigma\geq-1/3$, $|\alpha|\ll L^{-1}$. Here $\tau_0=1/3$ is admissible.

Now by Perron's formula
\begin{eqnarray}\label{48}
\sum_{k\leq y}\frac{d_r(km)g(k)}{\varphi(km)}&=&\frac{1}{2\pi i}\int_{1-iU}^{1+iU}H(s,\alpha)\frac{y^s}{s}ds\nonumber\\
&&\ +O\bigg(yd_r(m)\sum_{k=1}^{\infty}\frac{d_r(k)|g(k)|}{\varphi(km)k}\min\bigg(1,\frac{1}{U|\log y/k|}\bigg)\bigg).
\end{eqnarray}
By splitting the sum in the $O$-term into the ranges $[1,y/2)$, $[y/2,3y/2)$ and $[3y/2,\infty)$, we find that the sum is
\begin{displaymath}
\ll\frac{yd_r(m)F_{\tau_0}(m)}{Um}.
\end{displaymath}
We now move the line of integration in \eqref{48} to $\sigma=-1/4$ and use Cauchy's theorem. On RH
\begin{eqnarray*}
\zeta(s),\ \zeta(s)^{-1}\ll_\varepsilon(1+|t|)^\varepsilon\qquad(\sigma\geq1/2+\varepsilon,\ |s-1|\gg1),
\end{eqnarray*}
so by \eqref{47}, \eqref{50} and \eqref{49} we have
\begin{eqnarray*}
H(s,\alpha)\ll_\varepsilon\frac{U^\varepsilon d_r(m)F_{\tau_0}(m)}{m}
\end{eqnarray*}
on the new path of integration. So the contribution along the horizontal lines is
\begin{displaymath}
\ll_\varepsilon\frac{yd_r(m)F_{\tau_0}(m)}{Um},
\end{displaymath}
and that along the left edge is
\begin{displaymath}
\ll_\varepsilon\frac{U^\varepsilon d_r(m)F_{\tau_0}(m)}{y^{1/4}m}.
\end{displaymath}
Thus, taking $U=y\log y$ leads to
\begin{equation}\label{10}
\sum_{k\leq y}\frac{d_r(km)g(k)}{\varphi(km)}=\textrm{Res}_{s=0}\bigg(H(s,\alpha)\frac{y^s}{s}\bigg)+O\bigg(\frac{d_{r}(m)F_{\tau_0}(m)}{m}L^{-1}\bigg).
\end{equation}

To compute the residue, we use the Laurent expansion of each factor in
\begin{equation*}
H(s,\alpha)\frac{y^s}{s}=\zeta(1+s-i\alpha)^{-r}Z_{11}(s,\alpha)Z_2(s,\alpha)y^s\zeta(1+s)^rs^{-1}.
\end{equation*}
We have
\begin{eqnarray*}
\zeta(1+s)^rs^{-1}&=&s^{-r-1}(1+a_1s+a_2s^2+\ldots),\nonumber\\
y^s&=&1+(\log y)s+\frac{(\log y)^2}{2!}s^2+\ldots,\nonumber\\
\zeta(1+s-i\alpha)^{-r}&=&f(-i\alpha)+f'(-i\alpha)s+\frac{f''(-i\alpha)}{2!}s^2+\ldots,
\end{eqnarray*}
where we put $f(z)=\zeta(1+z)^{-r}$. It is standard to check that
\begin{equation*}
f^{(j)}(-i\alpha)=r(r-1)\ldots(r-j+1)(-i\alpha)^{r-j}+O(|\alpha|^{r-j+1})\qquad(0\leq j\leq r).
\end{equation*}
We also note that since $Z_{11}(s,\alpha)$ and $Z_2(s,\alpha)$ are analytic and uniformly bounded in $\sigma\geq-1/3$, $|\alpha|\ll L^{-1}$, by Cauchy's theorem
\begin{displaymath}
\bigg(\frac{\partial}{\partial s}\bigg)^{j}Z_{11}(0,\alpha)\ll1,\ \textrm{and }\bigg(\frac{\partial}{\partial s}\bigg)^{j}Z_{2}(0,\alpha)\ll d_r(m)/\varphi(m).
\end{displaymath}
The analyticity in $\alpha$ also implies that
\begin{displaymath}
Z_{11}(0,\alpha)=Z_{11}(0,0)+O(|\alpha|)=1+O(|\alpha|),
\end{displaymath}
and
\begin{displaymath}
Z_2(0,\alpha)=Z_2(0,0)+O(|\alpha|d_r(m)/\varphi(m))=d_r(m)/\varphi(m)(1+O(|\alpha|)).
\end{displaymath}
Thus the residue at $s=0$ is
\begin{eqnarray*}
\textrm{Res}_{s=0}&=&\sum_{u_1+u_2+u_3+u_4+u_5=r}\frac{a_{u_1}(\log y)^{u_2}f^{(u_3)}(-i\alpha)Z_{11}^{(u_4)}(0,\alpha)Z_{2}^{(u_5)}(0,\alpha)}{u_{1}!u_{2}!u_{3}!u_{4}!u_{5}!}\nonumber\\
&=&\sum_{u_2+u_3=r}\frac{(\log y)^{u_2}f^{(u_3)}(-i\alpha)}{u_{2}!u_{3}!}Z_{11}(0,\alpha)Z_{2}(0,\alpha)+O\bigg(\frac{d_r(m)}{\varphi(m)}L^{-1}\bigg)\nonumber\\
&=&\bigg(\sum_{0\leq u_2\leq y}\frac{(-i\alpha\log y)^{u_2}}{u_2!}\binom{r}{u_2}\bigg)\frac{d_r(m)}{\varphi(m)}+O\bigg(\frac{d_r(m)}{\varphi(m)}L^{-1}\bigg).
\end{eqnarray*}
Combining this with \eqref{10}, the lemma follows.
\end{proof}

\section{Proofs of Lemma 2.1 and Lemma 2.2}

From Montgomery-Vaughan's mean value theorem \textbf{\cite{MV1}} we have
\begin{equation*}
M_{1}=\int_{T}^{2T}|H_1({\scriptstyle{\frac{1}{2}}}+it)|^2dt\sim T\sum_{k\leq y}\frac{d_{r+1}(k)^2P_1[k]^2}{k}.
\end{equation*}
By Lemma 4.10,
\begin{equation*}
M_{1}\sim\frac{a_{r+1}T(\log y)^{(r+1)^{2}}}{\Gamma((r+1)^{2})}\int_{0}^{1}(1-x)^{(r+1)^{2}-1}P_1(x)^2dx.
\end{equation*}
This proves Lemma 2.1.

For Lemma 2.2, we first move the line of integration to $\Re{s}=a=1+L^{-1}$. As in Section 3, the contribution from the horizontal lines is $\ll yT^{1/4+\varepsilon}$. Now we have
\begin{equation*}
\zeta(s)=\sum_{n\leq T^{1/2}}\frac{1}{n^{s}}+O(L).
\end{equation*}
Hence
\begin{eqnarray}\label{21}
M_{12}=\int_{T}^{2T}\zeta({\scriptstyle{\frac{1}{2}}}+it)H_1({\scriptstyle{\frac{1}{2}}}-it)H_2({\scriptstyle{\frac{1}{2}}}+it)dt&=&\int_{T}^{2T}\sum_{n\leq T^{1/2}}\frac{1}{n^{s}}H_{1}(1-s)H_{2}(s)dt\nonumber\\
&&\!\!\!\!\!\!\!\!\!\!\!\!\!\!\!\!\!\!\!\!\!\!\!\!\!\!\!\!\!\!\!\!\!\!\!\!\!\!\!\!\!\!\!\!\!\!\!\!\!\!\!\!\!\!\!\!\!\!\!\!\!\!\!\!\!\!\!\!\!\!\!\!\!\!\!\!\!\!\!\!\!\!\!\!\!\!\!\!\!\!\!\!\!\!\!\!\!\!\!\!\!\! +O\bigg(L\int_{T}^{2T}|H_1({\scriptstyle{\frac{1}{2}}}-it)H_2({\scriptstyle{\frac{1}{2}}}+it)|dt\bigg)+O(yT^{1/4+\varepsilon}),
\end{eqnarray}
where $s=a+it$. Here the line of integration in the first $O$-term has been moved back to the $\tfrac{1}{2}$-line with an admissible error. By Cauchy's inequality and Lemma 2.1 this term is 
\begin{equation*}
\ll L\bigg(\int_{T}^{2T}|H_1({\scriptstyle{\frac{1}{2}}}+it)|^2\bigg)^{1/2}\bigg(\int_{T}^{2T}|H_2({\scriptstyle{\frac{1}{2}}}+it)|^2\bigg)^{1/2}\ll TL^{r^2+r+3/2}.
\end{equation*}
Furthermore from Montgomery-Vaughan's mean value theorem, the main term is asymptotic to
\begin{equation*}
T\sum_{\substack{h,k\leq y\\h=kn}}\frac{d_{r+1}(h)P_{1}[h]d_{r}(k)P_{2}[k]}{h}.
\end{equation*}
So, by Lemma 4.11,
\begin{equation*}
M_{12}\sim\frac{a_{r+1}T(\log y)^{(r+1)^2}}{\Gamma(r+1)\Gamma(r(r+1))}\int_{0}^{1}(1-x)^{r(r+1)-1}Q_{r}(x)P_{2}(x)dx.
\end{equation*}
This proves Lemma 2.2.

\section{Proof of Lemma 2.4}

We have
\begin{eqnarray*}
S_1&=&\sum_{T\leq\gamma\leq 2T}H_1(\rho+i\alpha)H_1(1-\rho-i\alpha)=\sum_{h,k\leq y}\frac{d_{r+1}(h)P_1[h]d_{r+1}(k)P_1[k]}{h^{1-i\alpha}k^{i\alpha}}\sum_{T\leq\gamma\leq 2T}\bigg(\frac{h}{k}\bigg)^\rho\\
&=&I+I_1+I_2,
\end{eqnarray*}
where $I$, $I_1$ and $I_2$ are the contributions of the terms $h=k$, $h>k$ and $h<k$, respectively.

In view of Lemma 4.10
\begin{equation}\label{60}
I=\frac{TL}{2\pi}\sum_{k\leq y}\frac{d_{r+1}(k)^2P_1[k]^2}{k} \sim \frac{a_{r+1}TL(\log y)^{(r+1)^{2}}}{2\pi\Gamma((r+1)^{2})}\int_{0}^{1}(1-x)^{(r+1)^{2}-1}P_1(x)^2dx.
\end{equation}
Next we note that $I_2=\overline{I_1}$. We obtain from Lemma 4.6 that
\begin{eqnarray*}
I_1&=&-\frac{T}{2\pi}\sum_{h,k\leq y}\frac{d_{r+1}(h)P_1[h]d_{r+1}(k)P_1[k]}{h^{1-i\alpha}k^{i\alpha}}\Lambda\bigg(\frac{h}{k}\bigg)+O\bigg(L\log L\sum_{k<h\leq y}\frac{d_{r+1}(h)d_{r+1}(k)}{h}\bigg)\\
&&\qquad +O\bigg(L\sum_{k<h\leq y}\frac{d_{r+1}(h)d_{r+1}(k)}{h\langle h/k\rangle}\bigg)+O\bigg(L\sum_{k<h\leq y}\frac{d_{r+1}(h)d_{r+1}(k)}{h\log h/k}\bigg).
\end{eqnarray*}
We denote these four terms by $I_{11}$, $I_{12}$, $I_{13}$ and $I_{14}$, respectively. We have
\begin{eqnarray*}
I_{11}=-\frac{T}{2\pi}\sum_{\substack{h,k\leq y\\h=kn}}\frac{\Lambda(n)d_{r+1}(h)P_1[h]d_{r+1}(k)P_1[k]}{hn^{-i\alpha}}.
\end{eqnarray*}
Using Lemma 4.11 we get
\begin{equation}\label{61}
I_{11}\sim-\frac{(r+1)a_{r+1}T(\log y)^{(r+1)^2+1}}{2\pi\Gamma((r+1)^2)}\int_{0}^{1}\int_{0}^{x}(1-x)^{(r+1)^2-1}y^{i\alpha t}P_1(x-t)P_1(x)dtdx.
\end{equation}
Combining \eqref{60} and \eqref{61}, we obtain the main term in Lemma 2.4.

We are left show that the error terms $I_{12}$, $I_{13}$ and $I_{14}$ are admissible. The bound for $I_{12}$ is trivial, 
\begin{equation*}
I_{12}\ll T^{\varepsilon}\sum_{h,k\leq y}\frac{1}{h}\ll yT^{\varepsilon}.
\end{equation*}
To estimate $I_{13}$, we write $h=uk+v$ where $|v/k|\leq\tfrac{1}{2}$. We observe that $\langle h/k\rangle=|v/k|$ if $u$ is a prime power and $v\ne0$, otherwise $\langle h/k\rangle\geq\tfrac{1}{2}$. So
\begin{displaymath}
I_{13}\ll T^{\varepsilon}\bigg(\sum_{uk\ll y}\sum_{1\leq v\leq k/2}\frac{1}{v}+\sum_{h,k\leq y}\frac{1}{h}\bigg)\ll yT^{\varepsilon}.
\end{displaymath}
Finally for $I_{14}$, we note that $\log h/k\geq\log h/(h-1)\gg 1/h$. So $I_{14}\ll y^2T^\varepsilon$. The proof is complete.

\section{Proof of Lemma 2.5}

\subsection{Evaluation of $J_1(H_1,H_2)$}

We truncate the Dirichlet series of the product of the first two terms in \eqref{3} at $T^{1/2}$,
\begin{eqnarray*}
\frac{\zeta'}{\zeta}(s-i\alpha)\zeta(s)&=&-\sum_{mn\leq T^{1/2}}\frac{\Lambda(m)}{m^{s-i\alpha}n^s}+O\bigg(\sum_{n>T^{1/2}}\frac{\log n}{n^{1+L^{-1}}}\bigg)\nonumber\\
&=&-\sum_{mn\leq T^{1/2}}\frac{\Lambda(m)}{m^{s-i\alpha}n^s}+O(L^2).
\end{eqnarray*}
Hence
\begin{eqnarray*}
J_1(H_1,H_2)&=&-\frac{1}{2\pi i}\int_{a+i(T+\alpha)}^{a+i(2T+\alpha)}\sum_{mn\leq T^{1/2}}\frac{\Lambda(m)}{m^{s-i\alpha}n^s}H_{1}(1-s)H_{2}(s)ds\nonumber\\
&&\qquad\qquad+O\bigg(L^2\int_{T+\alpha}^{2T+\alpha}|H_1(a+it)H_2(1-a-it)|dt\bigg).
\end{eqnarray*}
As before we can move the line of integration in the $O$-term to the $\tfrac{1}{2}$-line with an admissible error of size $O(yT^\varepsilon)$. The same argument as in \eqref{21} then implies that the $O$-term is $\ll TL^{r^2+r+5/2}$. From Montgomery-Vaughan's mean value theorem, the main term is asymptotic to
\begin{equation*}
-\frac{T}{2\pi}\sum_{\substack{h,k\leq y\\h=kmn}}\frac{\Lambda(n)d_{r+1}(h)P_{1}[h]d_{r}(k)P_{2}[k]}{hn^{-i\alpha}}.
\end{equation*}
Thus, by Lemma 4.11,
\begin{eqnarray}\label{29}
J_1(H_1,H_2)\sim-\frac{(r+1)a_{r+1}T(\log y)^{(r+1)^2+1}}{2\pi\Gamma(r+1)\Gamma(r(r+1))}\int_{0}^{1}\int_{0}^{x}(1-x)^{r(r+1)-1}y^{i\alpha t}Q_{r}(x-t)P_2(x)dtdx.
\end{eqnarray}

\subsection{Evaluation of $J_2(H_1,H_2)$}

We recall that
\begin{equation*}
J_2(H_1,H_2)=\frac{1}{2\pi i}\int_{a+i(T+\alpha)}^{a+i(2T+\alpha)}\chi(1-s)\frac{\zeta'}{\zeta}(s-i\alpha)\zeta(s)H_{1}(s)H_{2}(1-s)ds.
\end{equation*}
By Lemma 4.1 we obtain
\begin{equation*}
J_2(H_1,H_2)=\sum_{k\leq y}\frac{d_{r}(k)P_{2}[k]}{k}\sum_{kT/2\pi\leq h\leq kT/\pi}a(h)e(-h/k)+O(yT^{1/2+\varepsilon}),
\end{equation*}
where
\begin{equation*}
a(h)=-\sum_{\substack{nuv=h\\n\leq y}}d_{r+1}(n)P_{1}[n]\Lambda(u)u^{i\alpha}.
\end{equation*}
We write
\begin{equation*}
Q^{*}(s,\alpha,k)=\sum_{h=1}^{\infty}\frac{a(h)e(-h/k)}{h^s}.
\end{equation*}
From Perron's formula, we have
\begin{equation}\label{13}
\sum_{h\leq kT/2\pi}a(h)e(-h/k)=\frac{1}{2\pi i}\int_{a-iT}^{a+iT}Q^{*}(s,\alpha,k)\bigg(\frac{kT}{2\pi}\bigg)^s\frac{ds}{s}+O(kT^{\varepsilon}).
\end{equation}
Lemma 4.5 asserts that $Q^{*}(s,\alpha,k)$ has at most two poles in $\sigma>\tfrac{1}{2}$ at $s=1$ and $s=1+i\alpha$ (we are assuming that $\alpha\ne0$). Hence we move the line of integration in \eqref{13} to $\sigma=a_0=\tfrac{1}{2}+L^{-1}$ and obtain
\begin{eqnarray*}
\frac{1}{2\pi i}\int_{a-iT}^{a+iT}Q^{*}(s,\alpha,k)\bigg(\frac{kT}{2\pi}\bigg)^s\frac{ds}{s}&=&R_1+R_{1+i\alpha}\nonumber\\
&&\!\!\!\!\!\!\!\!\!\!\!\!\!\!\!\!\!\!\!\!\!\!\!\!\!\!\!\!\!\!\!\!\!\!\!\!\!\!\!\!\!\!\!\!\!+\frac{1}{2\pi i}\bigg(\int_{a-iT}^{a_0-iT}+\int_{a_0-iT}^{a_0+iT}+\int_{a_0+iT}^{a+iT}\bigg)Q^{*}(s,\alpha,k)\bigg(\frac{kT}{2\pi}\bigg)^s\frac{ds}{s},
\end{eqnarray*}
where $R_1$ and $R_{1+i\alpha}$ are the residues of the integrand at $s=1$ and $s=1+i\alpha$, respectively. By Lemma 4.5, the left edge of the contour contributes 
\begin{equation*}
\ll y^{1/2}T^{\varepsilon}(kT)^{a_0}\int_{-T}^{T}\frac{dt}{1+|t|}\ll yT^{1/2+\varepsilon}.
\end{equation*}
Also, the contribution along the horizontal lines is
\begin{equation*}
\ll y^{1/2}T^{\varepsilon}\frac{(kT)^{a}}{T}\ll y^{3/2}T^{\varepsilon}.
\end{equation*}
Thus
\begin{equation*}
\sum_{h\leq kT/2\pi}a(h)e(-h/k)=R_1+R_{1+i\alpha}+O(yT^{1/2+\varepsilon}+y^{3/2}T^{\varepsilon}).
\end{equation*}

We now compute the residues $R_1$ and $R_{1+i\alpha}$. Let $Q(s,\alpha,h/k)$ be as in Lemma 4.3. Then we have
\begin{equation*}
Q^{*}(s,\alpha,k)=\sum_{h\leq y}\frac{d_{r+1}(h)P_{1}[h]}{h^s}Q(s,\alpha,h/k).
\end{equation*}
Hence by Lemma 4.3(i), we obtain
\begin{equation*}
R_1=\frac{kT}{2\pi}\sum_{h\leq y}\frac{d_{r+1}(h)P_{1}[h]}{h}\times\left\{ \begin{array}{ll}
\frac{\zeta'}{\zeta}(1-i\alpha) & \textrm{if $K=1$}\\
\frac{-\log p}{p^{(1-i\alpha)\lambda}(1-p^{-1+i\alpha})} & \textrm{if $K=p^\lambda>1$}\\
0 & \textrm{otherwise,}
\end{array} \right.
\end{equation*}
where $K=k/(h,k)$. Also, by Lemma 4.3(ii) we have
\begin{equation*}
R_{1+i\alpha}=-\frac{1}{1+i\alpha}\bigg(\frac{kT}{2\pi}\bigg)^{1+i\alpha}\sum_{h\leq y}\frac{d_{r+1}(h)P_{1}[h]}{h^{1+i\alpha}}\frac{\zeta(1+i\alpha)}{K^{i\alpha}\varphi(K)}\prod_{p|K}(1-p^{i\alpha}).
\end{equation*}
Thus
\begin{eqnarray*}
J_2(H_1,H_2)&=&\frac{T}{2\pi}\frac{\zeta'}{\zeta}(1-i\alpha)\sum_{\substack{h,k\leq y\\h=kn}}\frac{d_{r+1}(h)P_{1}[h]d_{r}(k)P_{2}[k]}{h}\nonumber\\
&&\nonumber\\
&&\!\!\!\!\!\!\!\!\!\!\!\!\!\!\!\!\!\!\!\!\!\!\!-\frac{T}{2\pi}\sum_{1<p^\lambda\leq y}\frac{\log p}{p^{(1-i\alpha)\lambda}(1-p^{-1+i\alpha})}\sum_{\substack{(h,p)=1\\hk\leq y\\k\leq y/p^\lambda}}\frac{d_{r+1}(hk)P_{1}[hk]d_{r}(p^\lambda k)P_2[p^\lambda k]}{hk}\nonumber\\
&&\!\!\!\!\!\!\!\!\!\!\!\!\!\!\!\!\!\!\!\!\!\!\!-\frac{\zeta(1+i\alpha)}{1+i\alpha}\bigg(\frac{T}{2\pi}\bigg)^{1+i\alpha}\sum_{\substack{h,k\leq y}}\frac{d_{r+1}(h)P_{1}[h]d_{r}(k)P_{2}[k]}{hk^{-i\alpha}}\frac{\prod_{p|K}(1-p^{i\alpha})}{(hK)^{i\alpha}\varphi(K)}+O(yT^{1/2+\varepsilon}).
\end{eqnarray*}

We denote the three main terms by $J_{21}$, $J_{22}$ and $J_{23}$, respectively. The first expression follows from Lemma 4.11. By noting that $\zeta'(1-i\alpha)/\zeta(1-i\alpha)=(i\alpha)^{-1}+O(1)$, $J_{21}$ is asymptotic to
\begin{equation}\label{37}
\frac{a_{r+1}T(\log y)^{(r+1)^2}}{2\pi i\alpha\Gamma(r+1)\Gamma(r(r+1))}\int_{0}^{1}(1-x)^{r(r+1)-1}Q_{r}(x)P_{2}(x)dx.
\end{equation}

For the second expression, we first note that the contribution of the terms for which $\lambda\geq2$, or $p$ is a prime divisor of $h$ or $k$ is
\begin{equation*}
\ll T\sum_{\substack{h,k\leq y\\h=nk}}\frac{d_{r+1}(h)d_{r}(k)}{h}\ll TL^{(r+1)^2}.
\end{equation*}
Hence, we have, up to an error term of size $O(TL^{(r+1)^2})$,
\begin{equation*}
J_{22}=-\frac{rT}{2\pi}\sum_{p\leq y}\frac{\log p}{p^{1-i\alpha}-1}\sum_{\substack{h\leq y,k\leq y/p\\h=nk}}\frac{d_{r+1}(h)P_{1}[h]d_{r}(k)P_2[pk]}{h}.
\end{equation*}
By Lemma 4.11, the sum over $h$ and $k$ is
\begin{equation*}
\frac{a_{r+1}(\log y)^{(r+1)^2}}{\Gamma(r+1)\Gamma(r(r+1))}\int_{0}^{\frac{\log y/p}{\log y}}x^{r(r+1)-1}Q_{r}(1-x)P_{2}\bigg(\frac{\log y/p}{\log y}-x\bigg)dx,
\end{equation*}
up to an error term of size $O(L^{(r+1)^2-1})$. The contribution of this $O$-term to $J_{22}$ is $\ll TL^{(r+1)^2}$. Hence the leading term of $J_{22}$ is
\begin{equation*}
-\frac{ra_{r+1}T(\log y)^{(r+1)^2}}{2\pi\Gamma(r+1)\Gamma(r(r+1))}\sum_{p\leq y}\frac{\log p}{p^{1-i\alpha}-1}g(p),
\end{equation*}
where
\begin{equation*}
g(p)=\int_{0}^{\frac{\log y/p}{\log y}}x^{r(r+1)-1}Q_{r}(1-x)P_{2}\bigg(\frac{\log y/p}{\log y}-x\bigg)dx.
\end{equation*}
Now from the prime number theorem, it is standard to check that
\begin{equation*}
\sum_{p\leq y}\frac{\log p}{p^{1-i\alpha}-1}=\frac{y^{i\alpha}-1}{i\alpha}+O(1).
\end{equation*}
So by Stieltjes integration,
\begin{equation*}
\sum_{p\leq y}\frac{\log p}{p^{1-i\alpha}-1}g(p)=\int_{1}^{y}\frac{g(t)dt}{t^{1-i\alpha}}+O(1)=\log y\int_{0}^{1}g(y^t)y^{i\alpha t}dt+O(1).
\end{equation*}
Thus 
\begin{eqnarray}\label{34}
J_{22}\sim-\frac{ra_{r+1}T(\log y)^{(r+1)^2+1}}{2\pi\Gamma(r+1)\Gamma(r(r+1))}\int_{0}^{1}\int_{0}^{x}(1-x)^{r(r+1)-1}y^{i\alpha t}Q_{r}(x)P_{2}(x-t)dtdx.
\end{eqnarray}

We are left to evaluate $J_{23}$. Using the M\"obius inversion 
\begin{equation*}
f((h,k))=\sum_{\substack{m|h\\m|k}}\sum_{n|m}\mu(n)f\bigg(\frac{m}{n}\bigg),
\end{equation*}
the sum over $h$ and $k$ is
\begin{equation}\label{45}
\sum_{\substack{h,k\leq y}}\frac{d_{r+1}(h)P_{1}[h]d_{r}(k)P_{2}[k]}{hk^{-i\alpha}}\sum_{\substack{m|h\\m|k}}\sum_{n|m}\mu(n)\frac{\prod_{p|kn/m}(1-p^{i\alpha})}{(\frac{hkn}{m})^{i\alpha}\varphi(\frac{kn}{m})}.
\end{equation}
By writing $hm$ and $km$ for $h$ and $k$, respectively, the above expression is
\begin{equation*}
\sum_{m\leq y}\frac{1}{m}\sum_{h\leq y/m}\frac{d_{r+1}(hm)P_{1}[hm]}{h^{1+i\alpha}}\sum_{k\leq y/m}d_{r}(km)P_{2}[km]\sum_{n|m}\frac{\mu(n)}{n^{i\alpha}}\frac{\prod_{p|kn}(1-p^{i\alpha})}{\varphi(kn)}.
\end{equation*}
We let 
\begin{eqnarray*}
f(k)=\frac{\prod_{p|k}(1-p^{i\alpha})}{\varphi(k)}.
\end{eqnarray*}
It is standard to verify that $f(k)$ is multiplicative. Hence the sum over $n$ is 
\begin{eqnarray*}
\sum_{n|m}\frac{\mu(n)f(kn)}{n^{i\alpha}}=f(k)\prod_{p|m}\bigg(1-\frac{f(kp)}{f(k)p^{i\alpha}}\bigg)=f(k)\prod_{\substack{p|m\\p\nmid k}}\bigg(1-\frac{f(p)}{p^{i\alpha}}\bigg)\prod_{\substack{p|m\\p|k}}\bigg(1-\frac{\varphi(k)}{\varphi(kp)p^{i\alpha}}\bigg).
\end{eqnarray*}
This can be simplified further as
\begin{eqnarray*}
\frac{\prod_{p|k}(1-p^{i\alpha})}{\varphi(k)}\prod_{\substack{p|m\\p\nmid k}}\frac{p(1-p^{-(1+i\alpha)})}{p-1}\prod_{\substack{p|m\\p|k}}(1-p^{-(1+i\alpha)})=\frac{mA(m)g(k)}{\varphi(km)},
\end{eqnarray*}
where
\begin{eqnarray*}
A(m)=\prod_{p|m}(1-p^{-(1+i\alpha)})\qquad\textrm{and}\qquad g(k)=\prod_{p|k}(1-p^{i\alpha}).
\end{eqnarray*}
So \eqref{45} is equal to
\begin{equation*}
\sum_{m\leq y}A(m)\sum_{h\leq y/m}\frac{d_{r+1}(hm)P_{1}[hm]}{h^{1+i\alpha}}\sum_{k\leq y/m}\frac{d_{r}(km)g(k)P_{2}[km]}{\varphi(km)}.
\end{equation*}
By Stieltjes integration and Lemma 4.12, the sum over $k$ is
\begin{eqnarray*}
&&\frac{d_{r}(m)}{\varphi(m)}\bigg(\sum_{j=1}^{r}\binom{r}{j}\int_{1}^{y/m}\frac{(-i\alpha\log t)^{j}}{(j-1)!t\log t}P_2[mt]dt+P_2[m]\bigg)+O\bigg(\frac{d_{r}(m)F_{\tau_0}(m)}{m}L^{-1}\bigg)\nonumber\\
&=&\frac{d_{r}(m)}{\varphi(m)}\bigg(\sum_{j=1}^{r}\binom{r}{j}\frac{(-i\alpha\log y)^{j}}{(j-1)!}R_{j-1}[m]+P_2[m]\bigg)+O\bigg(\frac{d_{r}(m)F_{\tau_0}(m)}{m}L^{-1}\bigg).
\end{eqnarray*}
Using Lemma 4.7 and Lemma 4.8, the contribution of the $O$-term to $J_{23}$ is
\begin{equation*}
\ll TL^{r+1}\sum_{m\leq y}\frac{d_{r+1}(m)d_{r}(m)F_{\tau_0}(m)}{\varphi(m)}\ll TL^{(r+1)^2+\varepsilon}.
\end{equation*}
Now Lemma 4.7 gives
\begin{eqnarray*}
\sum_{h\leq y/m}\frac{d_{r+1}(hm)P_{1}[hm]}{h^{1+i\alpha}}&=&\frac{D_{r+1}(m)(\log y)^{r+1}}{\Gamma(r+1)}\nonumber\\
&&\!\!\!\!\!\!\!\!\!\!\!\!\!\!\!\!\!\!\!\!\!\!\!\!\!\!\!\!\!\!\!\!\!\!\!\!\!\!\!\!\!\!\!\!\!\int_{0}^{\frac{\log y/m}{\log y}}t^{r}y^{-i\alpha t}P_1\bigg(\frac{\log y/m}{\log y}-t\bigg)dt+O(d_{r+1}(m)F_{\tau_0}(m)L^{r}).
\end{eqnarray*}
Again, the contribution of this $O$-term to $J_{23}$ is $\ll TL^{(r+1)^2+\varepsilon}$. Thus, up to an error term of size $O(TL^{(r+1)^2+\varepsilon})$,
\begin{eqnarray*}
J_{23}&=&-\frac{\zeta(1+i\alpha)}{1+i\alpha}\bigg(\frac{T}{2\pi}\bigg)^{1+i\alpha}\frac{(\log y)^{r+1}}{\Gamma(r+1)}\sum_{m\leq y}\frac{D_{r+1}(m)d_{r}(m)A(m)}{\varphi(m)}\nonumber\\
&&\qquad\int_{0}^{\frac{\log y/m}{\log y}}t^{r}y^{-i\alpha t}P_1\bigg(\frac{\log y/m}{\log y}-t\bigg)\bigg(\sum_{j=1}^{r}\binom{r}{j}\frac{(-i\alpha\log y)^{j}}{(j-1)!}R_{j-1}[m]+P_2[m]\bigg)dt.
\end{eqnarray*}
By Lemma 14 and Stieltjes integration, the sum over $m$ is
\begin{eqnarray*}
&&\frac{a_{r+1}}{\Gamma(r(r+1))}\int_{1}^{y}\int_{0}^{\frac{\log y/x}{\log y}}\frac{(\log x)^{r(r+1)-1}}{x}t^{r}y^{-i\alpha t}P_1\bigg(\frac{\log y/x}{\log y}-t\bigg)\\
&&\qquad\qquad \bigg(\sum_{j=1}^{r}\binom{r}{j}\frac{(-i\alpha\log y)^{j}}{(j-1)!}R_{j-1}[x]+P_2[x]\bigg)dtdx+O(L^{r(r+1)-1}).\nonumber
\end{eqnarray*}
We note that
\begin{eqnarray*}
\frac{\zeta(1+i\alpha)}{1+i\alpha}\bigg(\frac{T}{2\pi}\bigg)^{1+i\alpha}=\frac{T^{1+i\alpha}}{2\pi i\alpha}+O(T).
\end{eqnarray*}
Hence, substituting $1-\log x/\log y$ by $x$ leads to
\begin{eqnarray}\label{39}
J_{23}&\sim&-\frac{a_{r+1}T^{1+i\alpha}(\log y)^{(r+1)^2}}{2\pi i\alpha\Gamma(r+1)\Gamma(r(r+1))}\int_{0}^{1}\int_{0}^{x}(1-x)^{r(r+1)-1}t^{r}y^{-i\alpha t}\nonumber\\
&&\qquad\qquad P_1(x-t)\bigg(\sum_{j=1}^{r}\binom{r}{j}\frac{(-i\alpha\log y)^{j}}{(j-1)!}R_{j-1}(x)+P_2(x)\bigg)dtdx.
\end{eqnarray}

\subsection{Evaluation of $J_3(H_1,H_2)$}

We will first consider
\begin{equation*}
J_4(t)=\frac{1}{2\pi i}\int_{a+i(t+\alpha)}^{a+i(2t+\alpha)}\zeta(1-s)H_{1}(s)H_{2}(1-s)ds.
\end{equation*}
As before, we move the line of integration to the $\tfrac{1}{2}$-line. The contribution along the horizontal lines is $O(yt^{1/4+\varepsilon})$. The integral along the left edge is
\begin{equation*}
\frac{1}{2\pi}\overline{\int_{t+\alpha}^{2t+\alpha}\zeta(\tfrac{1}{2}+it)H_1(\tfrac{1}{2}-it)H_2(\tfrac{1}{2}+it)dt}.
\end{equation*}
Hence, by Lemma 2.2 we have
\begin{eqnarray}\label{4}
J_4(t)&=&\frac{a_{r+1}t(\log y)^{(r+1)^2}}{2\pi\Gamma(r+1)\Gamma(r(r+1))}\int_{0}^{1}(1-x)^{r(r+1)-1}Q_{r}(x)P_{2}(x)dx\nonumber\\
&&\qquad\ +O(yt^{1/4+\varepsilon})+O(tL^{(r+1)^2-1}).
\end{eqnarray}

By Stirling's formula we have
\begin{equation*}
\frac{\chi'}{\chi}(\tfrac{1}{2}-it+i\alpha)=-\log\frac{t}{2\pi}+O(t^{-1})\qquad(t\geq1).
\end{equation*}
Hence
\begin{eqnarray*}
J_3(H_1,H_2)&=&-\int_{T}^{2T}\log\frac{t}{2\pi}J_{4}'(t)dt\nonumber\\
&&\!\!\!\!\!\!\!\!\!\!\!\!\!\!\!\!+O\bigg(\int_{T+\alpha}^{2T+\alpha}|\chi(1-a-it)\zeta(a+it)H_{1}(a+it)H_{2}(1-a-it)|\frac{dt}{t}\bigg).
\end{eqnarray*}
The integrand in the error term is $\ll yT^{-1/2+\varepsilon}$. So the $O$-term is bounded by $yT^{1/2+\varepsilon}$. Hence, integration by parts leads to
\begin{equation*}
J_3(H_1,H_2)=-(\log T)\big(J_{4}(2T)-J_4(T)\big)+O\bigg(\bigg|\int_{T}^{2T}\frac{J_{4}(t)}{t}dt\bigg|\bigg)+O(yT^{1/2+\varepsilon}).
\end{equation*}
In view of \eqref{4}, we deduce that
\begin{eqnarray*}
J_3(H_1,H_2)&\sim&-\frac{a_{r+1}TL(\log y)^{(r+1)^2}}{2\pi\Gamma(r+1)\Gamma(r(r+1))}\int_{0}^{1}(1-x)^{r(r+1)-1}Q_{r}(x)P_{2}(x)dx.
\end{eqnarray*}

This, \eqref{29}, \eqref{37}, \eqref{34} and \eqref{39} establish Lemma 2.5.

\section{Deduction of Theorem 1.1}

In this section, we will demonstrate how Theorem 1.1 follows from Lemmas 2.1--2.6. Our arguments show that we can choose $\vartheta=\tfrac{1}{2}-\varepsilon$. Hence Lemmas 2.1--2.3 give
\begin{equation}\label{40}
\mathscr{M}_1(H,T)\sim a_{r+1}T(\log y)^{(r+1)^2}U,
\end{equation}
where
\begin{eqnarray*}
U&=&\frac{1}{\Gamma((r+1)^2)}\int_{0}^{1}(1-x)^{(r+1)^2-1}P_1(x)^2dx\\
&&\qquad\qquad+\frac{2}{\Gamma(r+1)\Gamma(r(r+1))}\int_{0}^{1}(1-x)^{r(r+1)-1}Q_r(x)P_2(x)dx\\
&&\qquad\qquad+\frac{2}{\Gamma(r)^2\Gamma(r^2)}\int_{0}^{1}(1-x)^{r^2-1}(R_{r-1}(x)-R_r(x))R_{r-1}(x)dx.
\end{eqnarray*}

Now from Lemma 2.4,
\begin{eqnarray}\label{41}
\int_{-c/L}^{c/L}\sum_{T\leq\gamma\leq 2T}|H_1(\tfrac{1}{2}+i(\gamma+\alpha))|^2d\alpha\sim\frac{a_{r+1}T(\log y)^{(r+1)^2}}{\pi}V_1,
\end{eqnarray}
where
\begin{eqnarray*}
V_1=\frac{1}{\Gamma((r+1)^2)}\int_{0}^{1}(1-x)^{(r+1)^2-1}P_1(x)\bigg(cP_1(x)-2(r+1)\int_{0}^{x}\frac{\sin(\frac{ct}{2})P_1(x-t)}{t}dt\bigg)dx.
\end{eqnarray*}
Similarly, by Lemma 2.5 we have
\begin{eqnarray}\label{42}
2\Re\bigg(\int_{-c/L}^{c/L}\sum_{T\leq\gamma\leq 2T}\zeta H_2(\rho+i\alpha)H_1(1-\rho-i\alpha)d\alpha\bigg)\sim\frac{a_{r+1}T(\log y)^{(r+1)^2}}{\pi}V_2,
\end{eqnarray}
where
\begin{eqnarray*}
V_2&=&\frac{1}{\Gamma(r+1)\Gamma(r(r+1))}\int_{0}^{1}(1-x)^{r(r+1)-1}\bigg(2cQ_r(x)P_2(x)\\
&&\!\!\!\!\!\!\!\!\!\! -2(r+1)P_2(x)\int_{0}^{x}\frac{\sin(\frac{ct}{2})Q_r(x-t)}{t}dt-2rQ_r(x)\int_{0}^{x}\frac{\sin(\frac{ct}{2})P_2(x-t)}{t}dt\\
&&\!\!\!\!\!\!\!\!\!\! +P_2(x)\int_{0}^{x}\int_{-c/2}^{c/2}\frac{\sin((t-2)\eta)}{\eta}t^rP_1(x-t)d\eta dt\\
&&\!\!\!\!\!\!\!\!\!\! +\sum_{2j+1\leq r}\frac{(-1)^j}{(2j)!}\binom{r}{2j+1}R_{2j}(x)\int_{0}^{x}\int_{-c/2}^{c/2}\cos((t-2)\eta)\eta^{2j}t^rP_1(x-t)d\eta dt\\
&&\!\!\!\!\!\!\!\!\!\! +\sum_{2j+2\leq r}\frac{(-1)^{j+1}}{(2j+1)!}\binom{r}{2j+2}R_{2j+1}(x)\int_{0}^{x}\int_{-c/2}^{c/2}\sin((t-2)\eta)\eta^{2j+1}t^rP_1(x-t)d\eta dt\bigg)dx.
\end{eqnarray*}
Furthermore, we note that $\Re(i\alpha\log y)^j=(-1)^l(\alpha\log y)^{2l}$ for $j=2l$, and $\Re(i\alpha\log y)^j=0$ for $j=2l+1$. Hence Lemma 2.6 gives
\begin{eqnarray}\label{43}
\int_{-c/L}^{c/L}\sum_{T<\gamma\leq 2T}|\zeta H_2({\scriptstyle{\frac{1}{2}}}+i(\gamma+\alpha))|^2d\alpha\sim\frac{a_{r+1}T(\log y)^{(r+1)^2}}{\pi}V_3,
\end{eqnarray}
where
\begin{eqnarray*}
V_3=\frac{1}{\Gamma(r)^2\Gamma(r^2)}\sum_{j=1}^{\infty}\frac{(-1)^jc^{2j+1}}{2^{2j-1}(2j+1)}\int_{0}^{1}(1-x)^{r^2-1}B(r,\tfrac{1}{2},2j;x)dx.
\end{eqnarray*}
Here $B(r,\theta,j;x)$ is defined as in Lemma 2.6.

Combining \eqref{40}, \eqref{41}, \eqref{42} and \eqref{43} we obtain that
\begin{eqnarray*}
h(c)=\frac{1}{\pi}\frac{V_1+V_2+V_3}{U}+o(1).
\end{eqnarray*}
Consider the polynomials $P_1(x)=\sum_{j\leq M}c_jx^j$ and $P_x(x)=\sum_{j\leq M}d_jx^j$. Choosing $r=2$, $M=10$ and running Mathematica's Minimize command, we obtain $\lambda>3.033$. Precisely, with 
\begin{eqnarray*}
P_1(x)&=&-3+97x-1730x^2+14830x^3-70248x^4+172217x^5-154805x^6-109555x^7\\
&&\qquad\qquad+188895x^8+130288x^9-186298x^{10}
\end{eqnarray*}
and
\begin{eqnarray*}
P_2(x)&=&-258+9245x-96770x^2+428888x^3-856147x^4+592829x^5+169210x^6\\
&&\qquad\qquad+94624x^7-716274x^8+230263x^9+154420x^{10},
\end{eqnarray*}
we have
\begin{eqnarray*}
h(3.033\pi)=0.998885\ldots<1.
\end{eqnarray*}
This and \eqref{51} complete the proof of the theorem.

\end{document}